\documentclass[a4paper,11pt]{article}
\usepackage{makeidx}
\usepackage[mathscr]{eucal}
\usepackage[dvips]{color}
\usepackage{amssymb, amsfonts, amsmath, amsthm, graphicx}
\usepackage[notcite, notref]{showkeys}
\usepackage{here}
\usepackage{tasks}
\usepackage{caption}
\usepackage{ esint }

\DeclareMathOperator{\sech}{sech}

\newtheorem{definition}{Definition}[section]
\newtheorem{proposition}{Proposition}[section]

\newtheorem{theorem}{Theorem}[section]

\newtheorem{lemma}[proposition]{Lemma}
\newtheorem{remark}{Remark}[section]

\newcommand{\eps}{\varepsilon}

\newcommand{\R}{\mathbb{R}}

\newcommand{\T}{\mathbb{T}}

\newcommand{\boO}{\mathcal{O}}

\numberwithin{equation}{section}

\author{}
\date{}
\begin{document}
\title{Normal form for transverse instability of gZK equation for the line soliton with nearly critical speed.}
\author{Yakine Bahri and Hichem Hajaiej}

\maketitle
\begin{abstract}
In this paper, we study the transverse instability of generalized Zakharov-Kuznetsov equation for the line soliton with critical speed. We derive and justify a normal form reduction for a bifurcation problem of the stationary nonlinear KdV equation on the product space $\R \times \T$.
\end{abstract}
\section{Introduction}

We consider the generalized Zakharov--Kuznetsov (gZK) equation with a general power non-linearity
\begin{equation}\label{ZKk}
u_t +\partial_x ( u^{k+1}) + u_{xxx} + u_{xyy} = 0, \quad (x,y)\in \R\times \T,
\end{equation}
where $k\geq 2$.
It is well-known that the following quantities are conserved for smooth solutions of the ZK equation :
\begin{equation} \label{M}
M(u)=\int u ^2dxdy,
\end{equation}
and
\begin{equation} \label{H}
H(u)=\frac12 \int \big( |\nabla u|^2-\frac1{k+2}u^{k+2}\big)dxdy.
\end{equation}

The gZK equation admits a traveling wave solution. A traveling wave with speed $c$, is a particular solution of the form $u(t,x,y)=u_c(x-ct,y)$ where $u_c$ is a non-trivial solution to the following stationary equation,
\begin{equation}
\label{TW}
- \partial_{xx} u_c - \partial_{yy} u_c + c u_c
- u_c^{3} = 0 \qquad  \mbox{in} \quad  \mathbb{R} \times \mathbb{T}.
\end{equation}

The main part of our work covers the instability of the line soliton over the 2D flow. We recall the definition of the orbital stability.
\begin{definition}
We say that a solitary wave $Q(x-c t, y)$ is orbitally stable in $H^{1}\left(\mathbb{R} \times \mathbb{T}_{L}\right)$ if for any $\varepsilon>0$ there exists $\delta>0$ such that for all initial data $u_{0} \in H^{1}\left(\mathbb{R} \times \mathbb{T}_{L}\right)$ with $\left\|u_{0}-Q\right\|_{H^{1}}<\delta$, the solution $u(t)$ of $(1.1)$ with $u(0)=u_{0}$ exists globally in positive time and satisfies
$$
\sup _{t>0} \inf _{\left(x_{0}, y_{0}\right) \in \mathbb{R} \times \mathbb{T}_{L}}\left\|u(t, \cdot, \cdot)-Q\left(\cdot-x_{0}, \cdot-y_{0}\right)\right\|_{H^{1}}<\varepsilon
$$
Otherwise, we say the solitary wave $Q(x-c t, y)$ is orbitally unstable in $H^{1}\left(\mathbb{R} \times \mathbb{T}_{L}\right)$.
\end{definition}
In the multidimensional case, de Bouard \cite{deB} showed the orbital stability of positive solitary waves of the generalized Zakharov–Kuznetsov equation  while assuming the well-posedness on the energy space. For the two dimensional case, Côte, Muñoz, Pilod and Simpson \cite{Munoz} have proved the asymptotic stability of positive solitary waves and multi-solitary waves of the Zakharov–Kuznetsov equation using the argument of Martel and Merle \cite{MM1}. A numerical study for two-dimensional solitary wave interactions and the formation of singularities in the modified Zakharov–Kuznetsov (mZK) equation was considered in \cite{SB}. This is a generalization of the Korteweg–deVries equation discussed in \cite{Pel1}.

Note that, $u_c(x,y)=u_c(x)$ is also a solution in the energy space. In this case, $u_{c}$ is the unique positive solution to
\begin{equation} \label{line-sp}
- \partial_{xx} u_{c} + c u_{c} - u_{c}^{k+1} = 0 \qquad \mbox{in $\mathbb{R}$}.
\end{equation}
We recall that the explicit form of the solution is given by
\begin{equation}
\label{line-soliton}
u_{c}(x) :=
c^{\frac{1}{k}} \left(
\frac{k+2}{2}
\right)^{\frac{1}{k}}
\sech^{\frac{2}{k}}
\left(\frac{k}{2} \sqrt{c} \ x\right).
\end{equation}

The stability of this traveling wave with respect to the KdV flow (the one-dimensional version of ZK) has been fully described in the literature. The orbital stability was proven to occur only for $k<4$ (see \cite{Benjamin, pw2} for more details). In addition, Martel and Merle proved in \cite{MM1} the asymptotic stability in the energy space. The instability holds for any $k\geq 4$. We refer the reader to \cite{BSW} for $k>4$ and to \cite{MM2, MMR, MMR1, M2} for the $L^2$ critical case $k=4$. We conclude that the line traveling wave $u_c$ is unstable under the ZK flow perturbation for $k\geq 4$. For this, we restrict our study to $2\leq k<4$.

For $k=1$, Yamazki studied in \cite{YY} the orbital and asymptotic transverse stability. He proved that the results depend on the speed of the traveling wave. More precisely, he showed that only traveling waves with subcritical as well as critical speed are asymptotically stable, and are unstable with the super-critical speed under a $2$D perturbation. He used the method of Evans’ function and the argument of Rousset and Tzvetkov \cite{RT0}. The asymptotic stability for orbitally stable line solitary waves of Zakharov–Kuznetsov equation was proved by using the argument of Martel and Merle \cite{MM1} and a Liouville type theorem combined with a modified virial type estimate since the linearized operator of the stationary equation is degenerate. In \cite{YY2}, Yamazaki constructed center stable manifolds around unstable line solitary waves with critical speed. To recover the degeneracy of the linearized operator around line solitary waves with critical speed, the auther proved the stability condition of the center stable manifold for critical speed by applying to the estimate of the 4th order term of a Lyapunov function in \cite{YY} and \cite{YY1}. In addition, Pelinovsky proved in \cite{Pel} the same results using a normal form argument. The normal form was derived by means of
symplectic projections and near-identity transformations. The justification of this normal form is provided with an energy method. More precisely; the main result of the work was to derive and to justify the first-order differential equation
\begin{equation}
\label{n-form}
\frac{d b}{d t}=\lambda^{\prime}\left(c_{*}\right)\left(c-c_{*}\right) b+\gamma|b|^{2} b, \quad t>0
\end{equation}

where $\lambda^{\prime}\left(c_{*}\right)>0, \gamma<0$ are real-valued numerical coefficients, $c_{*}$ is the critical speed of the line soliton, $c \in \mathbb{R}$ depends on the initial conditions, and $b(t): \mathbb{R}_{+} \rightarrow \mathbb{C}$ is an amplitude of transverse perturbation. The above differential equation describes the nonlinear dynamics of a small transverse perturbation of a fixed period to the line soliton with a nearly critical speed $c_{*}$ and is referred to as "normal form for transverse instability of the line soliton with a nearly critical speed of propagation".

In this paper, our main goal is to derive the normal form \eqref{n-form} corresponding to \eqref{ZKk}. We will show that in our case that $\gamma >0$ which implies instability of the line soliton with a nearly critical speed $c_{*}$. In the Appendix, we will prove a conjecture stated by Pelinovsky in \cite{Pel} for the quadratic case ($k=1$).

\subsection{Preliminaries}
For the sake of simplicity, we will restrict our self to $k=2$ while tackling the transverse instability. Note that, our argument applies to any $k\in[2,4)$. Let us consider
\begin{equation}\label{ZK}
u_t +\partial_x ( u^{3}) + u_{xxx} + u_{xyy} = 0, \quad (x,y)\in \R\times \T,
\end{equation}
The one-dimensional linearized operator around the line soliton is given by
\begin{equation}
\label{def:Lc}
L_c := -\partial^2_{x} +  c - 3 u_c^2 .
\end{equation}
The Schr\"{o}dinger operator $L_c : H^2(\mathbb{R}) \to L^2(\mathbb{R})$ is known \cite{LL} to have the kernel spanned by $u'_c$, the essential
spectrum located on $[c,\infty)$ and a simple isolated negative eigenvalue
$$ \lambda = -3c,  \quad \psi = u_c^2.$$

Let us consider the $2\pi$-periodic transverse perturbation to the line solitons (\ref{line-soliton}).
The linearized solution around the soliton $u(x,t) = u_c(\xi) + U(\xi) e^{\lambda t + i k y}$ with $\xi:=x-ct$
and $k \in \mathbb{Z}$ provided the following spectral problem at the linear level
\begin{equation}
\label{lin-k}
\partial_{\xi} (L_c + k^2) U = \lambda U, \quad k \in \mathbb{Z},
\end{equation}
where $\partial_{\xi} L_c : H^3(\mathbb{R}) \to L^2(\mathbb{R})$ is the linearized operator for the gKdV equation.

We recall that the spectrum of
$\partial_{\xi} L_c : H^3(\mathbb{R}) \to L^2(\mathbb{R})$ consists of a double zero eigenvalue and a continuous spectrum on $i \mathbb{R}$ (see  \cite{pw1} for more details). The double zero eigenvalue is associated with the following Jordan block of the operator $\partial_{\xi} L$:
\begin{equation}
\label{Kernel-double}
\partial_{\xi} L_c \partial_{\xi} u_c = 0, \quad \partial_{\xi} L_c \partial_c u_c = - \partial_{\xi} u_c,
\end{equation}
where the derivatives of $u_c$ in $\xi$ and $c$ are exponentially decaying functions of $\xi$. The following lemma characterizes the spectral problem (\ref{lin-k}) for any $k \in \mathbb{N}$.
\begin{lemma}
For any $n \in \mathbb{N}$, the spectral problem (\ref{lin-k})
has a pair of real eigenvalues $\pm \lambda_k(c)$
if $c > c_n : = \frac{n^2}{3}$. No eigenvalues with ${\rm Re}(\lambda) \neq 0$
exist if $c \in (0,c_n)$.
\label{theorem-spectral}
\end{lemma}

Note that, it is possible to show nonlinear orbital stability of
the line soliton with $c \in (0,c_*)$ and nonlinear instability of the line soliton using energy method with $c > c_*$, where
$$c_* := \min\limits_{n \in \mathbb{N}} c_n \equiv \frac{1}{3}.$$

In this work, we will consider the critical speed case i.e. $c = c_*$. We first study the spectral problem \eqref{lin-k} in the exponentially weighted space ( see \cite{pw1,pw2} for more details)
\begin{equation}
\label{exp-space}
H^s_{\mu}(\mathbb{R}) =
\left\{ u \in H^s_{\rm loc}(\mathbb{R}) : \quad e^{\mu \xi} u \in H^s(\mathbb{R}) \right\}, \quad s \geq 0, \quad \mu > 0.
\end{equation}

The following lemma reviews the spectral property of  \eqref{lin-k}.

\begin{lemma}
There is $\mu_0 > 0$ such that for every $\mu \in (0,\mu_0)$, the spectral problem (\ref{lin-k})
with $k = 1$ and $c = c_*$ considered in $L^2_{\mu}(\mathbb{R})$ admits a simple zero eigenvalue with the eigenfunction
$\psi_* \in H^3_{\mu}(\mathbb{R})$ and the adjoint eigenfunction $\eta_* \in H^3_{-\mu}(\mathbb{R})$, where
\begin{equation}
\label{limiting-function}
\psi_*(\xi) = 2c_* {\rm sech}^2(\sqrt{c_*} \xi), \quad
\eta_*(\xi) = 2c_* \int_{-\infty}^{\xi} {\rm sech}^2(\sqrt{c_*} \xi') d\xi',
\end{equation}
hence,
\begin{equation}
\label{limiting-coefficient}
\langle \eta_*, \psi_* \rangle_{L^2} = \frac{(2c_*)^2}{2} \left( \int_{\mathbb{R}} {\rm sech}^2(\sqrt{c_*} \xi) d \xi \right)^2
= 8 c_*.
\end{equation}
Moreover, for a given $\mu \in (0,\mu_0)$, there exists an interval $(c_-,c_+)$ with
 $c_- < c_* < c_+$ such that the spectral problem (\ref{lin-k})
with $k = 1$ and $c \in (c_-,c_+)$ considered in $L^2_{\mu}(\mathbb{R})$
admits a small eigenvalue $\lambda(c)$, where the mapping $c \mapsto \lambda$ is smooth and
is given by
\begin{equation}
\label{lambda-expansion}
\lambda(c) = \lambda'(c_*) (c-c_*) + \mathcal{O}((c-c_*)^2) \quad \mbox{\rm as} \quad c \to c_*,
\end{equation}
with
\begin{equation}
\label{derivative-eigenvalue}
\lambda'(c_*) = 2 \sqrt{c_*}.
\end{equation}
\label{theorem-bifurcation}
\end{lemma}

\subsection{Instability result}
 We use a modulation argument to construct a translation parameter $a(t)$ and a speed parameter $c(t)$ of the line solitons as well as its perturbation $\tilde{u}(t)$ defined in $H^{1}(\mathbb{R} \times \mathbb{T}) \cap H_{\mu}^{1}(\mathbb{R} \times \mathbb{T})$ for every $t \in [0,T)$. This allows us to consider the traveling coordinate $\xi=x- a(t)$ and use the decomposition
$$u(x, y, t)=u_{c(t)}(\xi)+\tilde{u}(\xi, y, t), \quad \xi=x-4 a(t)
$$
The time evolution of the varying parameters $a(t)$ and $c(t)$ and the perturbation term $\tilde{u}(t)$ are to be found from the evolution problem
$$
\tilde{u}_{t}=\partial_{\xi}\left(L_{c}-\partial_{y}^{2}+4(\dot{a}-c)\right) \tilde{u}+4(\dot{a}-c) \partial_{\xi} u_{c}-\dot{c} \partial_{c} u_{c}-6 \partial_{\xi} \tilde{u}^{2}
$$
where the differential expression for $L_{c}$ is the linearized operator around the line soliton (see \eqref{def:Lc} for more details).

The following theorem represents the normal form for transverse instability of the line soliton with a nearly critical speed of propagation.

\begin{theorem}
 Consider the Cauchy problem for the evolution equation \eqref{ZKk} with
$$\tilde{u}(0) \in H^{1}(\mathbb{R} \times \mathbb{T}) \cap H_{\mu}^{1}(\mathbb{R} \times \mathbb{T}),
$$
where $\mu>0$ is sufficiently small. There exist $\varepsilon_{0}>0$ and $C_{0}>0$ such that if the initial data satisfy the bound
$$
\left\|\tilde{u}(0)-2 \varepsilon \cos (y) \psi_{*}\right\|_{H^{1}(\mathbb{R} \times \mathbb{T}) \cap H_{\mu}^{1}(\mathbb{R} \times \mathbb{T})}+\left|c(0)-c_{*}\right| \leqslant \varepsilon^{2},
$$
for some $\varepsilon \in\left(0, \varepsilon_{0}\right)$, then there exist unique functions $a, b, c \in C^{1}\left(\mathbb{R}_{+}\right),$ and the unique solution
$$
\tilde{u}(t) \in C\left(\mathbb{R}_{+} ; H^{1}(\mathbb{R} \times \mathbb{T}) \cap H_{\mu}^{1}(\mathbb{R} \times \mathbb{T})\right)
$$
of the evolution equation \eqref{ZKk} satisfying the bound
$$\left\|\tilde{u}(t)-\left(b(t) e^{i y}+\bar{b}(t) e^{-i y}\right) \psi_{*}\right\|_{H^{1}(\mathbb{R} \times \mathbb{T})}+\left|c(t)-c_{*}\right|+|\dot{a}(t)-c(t)| \leqslant C_{0} \varepsilon^{2}, \quad t \in \mathbb{R}_{+}.$$
Furthermore, the function $b(t)$ satisfies the normal form
$$
\dot{b}=\lambda^{\prime}\left(c_{*}\right)\left(c_{+}-c_{*}\right) b+\gamma|b|^{2} b, \quad t \in \mathbb{R}_{+},
$$
with $b(0)=\varepsilon$ and $c_{+} \in \mathbb{R}$ satisfying $\left|c_{+}-c_{*}\right| \leqslant C_{0} \varepsilon^{2}$, where $\lambda^{\prime}\left(c_{*}\right)>0$ is given by (2.13) and $\gamma>0$ is a specific numerical coefficient given by \eqref{gamma} below. Consequently, $|b(t)| >  \frac{C_{0}}{\varepsilon}$ for every $t \in \mathbb{R}_{+}$.

\end{theorem}

\section{Transverse Instability}

\subsection{Transversely modulated solitary waves}

In this section, we construct a solution to \eqref{TW} which bifurcates from the line soliton with critical speed. We define the space of even functions
both in $\xi$ and $y$:
\begin{equation}
\label{even-space}
H^s_{\rm even} = \left\{ u \in H^s(\mathbb{R} \times \mathbb{T}) : \quad u(-\xi,y) = u(\xi,y) = u(\xi,-y) \right\}, \quad s \geq 0.
\end{equation}
The following lemma describes this bifurcation.

\begin{lemma}
\label{lemma-modulated-wave}
There exists $c_+ > c_*$ such that for every $c \in (c_*,c_+)$,
the nonlinear elliptic problem (\ref{TW}) has a nontrivial solution $u_b$ in $H^2_{\rm even}$
in addition to the line soliton (\ref{line-soliton}). The solution $u_b$ is expressed by
the expansion
\begin{equation}
\label{LS-decomposition}
u_b(\xi,y) = u_{c_*}(\xi) +  b \cos(y) \psi_*(\xi) + \tilde{u}_b(\xi,y),
\end{equation}
where $b \in \mathbb{R}$ is a nonzero root of the algebraic equation
\begin{equation}
\label{normal-form-static}
\alpha (c-c_*) b + \beta b^3 = 0
\end{equation}
and $\tilde{u}_b \in H^2_{\rm even}$ satisfies the bound $\| \tilde{u}_b \|_{H^2} \leq A b^2$
for a positive constant $A$ independently of $b$ and $c$. Here
$$
\alpha = 8 (c_*)^{\frac{3}{2}} > 0
$$
and $\beta < 0$ is a numerical coefficient given by (\ref{gamma-static}) below.
\end{lemma}

\begin{proof}
The proof is similar to the one of Lemma 1.3 in \cite{Pel}
 and relies on the method of Lyapunov--Schmidt reduction \cite{Nirenberg}.
We denote by $v_1(\xi,y) := \cos(y) \psi_*(\xi)$ the eigenfunction of the kernel of $L_{c_*} - \partial_y^2 : H^2_{\rm even} \to L^2_{\rm even}$. Let $c = c_* + \delta$ with $\delta \in \mathbb{R}$ being
sufficiently small. We plug the decomposition (\ref{LS-decomposition}) in \eqref{TW}, to obtain the projection equations
of the Lyapunov--Schmidt reduction method:
\begin{equation}
\label{reduction-1}
(L_{c_*} - \partial_y^2 +  \delta) \tilde{u}_b = \Pi \tilde{F},
\end{equation}
with
\begin{equation}
\label{expression-F}
\tilde{F} := -  \delta u_{c_*} -  \delta b v_1 + ( b v_1 + \tilde{u}_b)^3 + 3 u_{c_*} ( b v_1 + \tilde{u}_b)^2
\end{equation}
% the orthogonality condition $\langle v_1, \tilde{u}_b \rangle_{L^2(\mathbb{R} \times \mathbb{T})} = 0$,
and $\Pi$ is an orthogonal projection operator in $L^2(\mathbb{R} \times \mathbb{T})$ on the subspace spanned by $v_1$.
%, the correction term $\tilde{u}_b \in H^2_{\rm even}$ and the parameter $b \in \mathbb{R}$ are defined
This means that
$$
\frac{1}{2\pi} \langle v_1, \tilde{F} \rangle_{L^2(\mathbb{R} \times \mathbb{T})} = 0,
$$
which yields
\begin{equation}
\label{reduction-2} -  \frac{\delta}{2} b \| \psi_* \|^2_{L^2(\mathbb{R})} +
\frac{3}{8}b^3 \| \psi_* \|^4_{L^4(\mathbb{R})} + \frac{3}{\pi} b \langle v_1^2, u_{c_*} \tilde{u}_b \rangle_{L^2(\mathbb{R} \times \mathbb{T})} + \boO(b^5) = 0.
\end{equation}
Since the kernel of the linear operator
$L_{c_*} - \partial_y^2 : H^2_{\rm even} \to L^2_{\rm even}$ is one-dimensional, we know that there are $B > 0$ and $\delta_0 > 0$ such that
\begin{equation}
\label{inverse-operator}
\| \Pi (L_{c_*} - \partial_y^2 + 4 \delta)^{-1} \Pi \|_{L^2_{\rm even} \to L^2_{\rm even}} \leq B, \quad \forall |\delta| < \delta_0.
\end{equation}
Thus, for every small $b \in \mathbb{R}$ and small $\delta \in \mathbb{R}$,
we apply the fixed-point argument to
solve equation (\ref{reduction-1}) with (\ref{expression-F}) in $H^2_{\rm even}$
and to obtain a unique $\tilde{u}_b \in H^2_{\rm even}$ satisfying the bound
\begin{equation}
\label{reduction-2a}
\| \tilde{u}_b \|_{H^2} \leq A (|\delta| + b^2),
\end{equation}
where the positive constant $A$ is independent of $\delta$ and $b$. Now, we shall derive the algebraic equation (\ref{normal-form-static}). To do so, we perform a near-identity transformation
\begin{equation}
\label{v-component}
\tilde{u}_b(\xi,y) =  b^2 \cos(2y) w_2(\xi) + b^2 w_0(\xi) + \delta \partial_c u_{c_*}(\xi) + \tilde{w}_b(\xi,y),
\end{equation}
where
%\begin{equation}
%\label{explicit-solution-2}
%\partial_c u_{c_*}(\xi) = {\rm sech}^2(\sqrt{c_*} \xi) - \sqrt{c_*} \xi \tanh(\sqrt{c_*} \xi) {\rm sech}^2(\sqrt{c_*} \xi),
%\end{equation}
$w_0$ and $w_2$ are given by
\begin{equation}
\label{inhom-eq-1}
L_{c_*} w_0 = \frac{3}{2}  u_{c_*} \psi_*^2
\end{equation}
and
\begin{equation}
\label{inhom-eq-2}
(L_{c_*} + 4) w_2 = \frac{3}{2}  u_{c_*} \psi_*^2.
\end{equation}
On the other hand, $\tilde{w}_b$ satisfies the transformed equation
\begin{eqnarray}
\label{reduction-3}
(L_{c_*} - \partial_y^2 +  \delta) \tilde{w}_b = \Pi \tilde{G},
\end{eqnarray}
with
\begin{eqnarray*}
\tilde{G} :=  -  \delta^2 \partial_c u_{c_*}-  \delta b^2 \cos(2y) w_2 -  \delta b^2 w_0-  \delta b v_1 + 6 b u_{c_*}v_1 \tilde{u}_b + 3 u_{c_*}\tilde{u}_b^2++ ( b v_1 + \tilde{u}_b)^3 
\end{eqnarray*}
for $\tilde{u}_b$ is defined by (\ref{v-component}). Similarly to the above competition, for every small $b \in \mathbb{R}$ and $\delta \in \mathbb{R}$,
there exists a unique solution $\tilde{w} \in H^2_{\rm even}$ of equation (\ref{reduction-3}) satisfying the bound
\begin{equation}
\label{reduction-3a}
\| \tilde{w} \|_{H^2} \leq A (\delta^2 + |\delta| |b| + |b|^3),
\end{equation}
where the positive constant $A$ is independent of $\delta$ and $b$.

Substituting the near-identity transformation (\ref{v-component}) into the bifurcation equation
(\ref{reduction-2}) and using the bound (\ref{reduction-3a})
for the component $\tilde{w} \in H^2_{\rm even}$, we rewrite
(\ref{reduction-2}) in the equivalent form
\begin{equation}
\label{reduction-4}
\alpha \delta b + \beta b^3 + \mathcal{O}(\delta b^2,b^4) = 0,
\end{equation}
where we have introduced numerical coefficients $\alpha$ and $\beta$ as follows:
\begin{equation}
\label{alpha-static}
2\alpha :=  -\| \psi_* \|^2_{L^2} + 6 \langle \psi_*^2,u_{c_*} \partial_c u_{c_*} \rangle_{L^2}
\end{equation}
and
\begin{equation}
\label{gamma-static}
\beta := \frac{3}{4} \langle \psi_*^2u_{c_*}, 2w_0 + w_2 \rangle_{L^2}+\frac{3}{8} \| \psi_* \|^4_{L^4(\mathbb{R})} .
\end{equation}
%We have also removed $\mathcal{O}(\delta^2)$ from the remainder term in (\ref{reduction-4}), because
%the bifurcation equation (\ref{reduction-2}) is identically satisfied
%in the case of the line soliton with $b = 0$ for every small $\delta \in \mathbb{R}$.

Plugging (\ref{line-soliton}) and (\ref{limiting-function}) into \eqref{alpha-static}, we get
$$\alpha = -\frac{1}{2} \langle \psi_*, L_{c_*}' \psi_* \rangle_{L^2}   = 8 (c_*)^{\frac{3}{2}},$$
where $L_{c_*}':= 1-6u_{c_*} \partial_c u_{c_*}$. On the other hand, we are not able to provide an explicit form to $\beta$. This is why we will determine a bound to it. We first compute $w_0$ using \eqref{inhom-eq-1}. More precisely, we have
\begin{equation}
\label{explicit-solution-1}
w_0(\xi) =\frac{1}{2} \left( u_{c_*}^3- 4 c_* u_{c_*}\right).
\end{equation}
For $w_2$, since $L_{c_*} + 4 : H^2(\mathbb{R}) \to L^2(\mathbb{R})$ is strictly positive, we just write equation (\ref{inhom-eq-2}) as $w_2= \frac{3}{2} (L_{c_*} + 4)^{-1} ( u_{c_*} \psi_*^2).$ . Thus, using the bound $\|(L_{c_*} + 4)^{-1} \| \leq \frac{1}{3}$, we obtain
\begin{equation}
\label{sign:beta}
\beta \leq \frac{3}{2} \langle \psi_*^2 u_{c_*}, w_0\rangle_{L^2}+\frac{3}{8} \| \psi_* \|^4_{L^4(\mathbb{R})}  +\frac{1}{4} \|\psi_*^2 u_{c_*}\|^2_{L^2} = \frac{2^{6}}{315}(c_*)^{\frac{7}{2}} \left(-45+\frac{32}{3}\right) <0.
\end{equation}
Here, we have used
\begin{align*}
\|\psi_*\|^4_{L^4}=\frac{2^9}{35} (c_*)^\frac{7}{2}, \quad  \langle \psi_*^2u_{c_*}, w_0 \rangle_{L^2}=-\frac{2^{10}}{105} (c_*)^\frac{7}{2} \quad  \|\psi_*^2 u_{c_*}\|^2_{L^2} =\frac{2^{13}}{315} (c_*)^\frac{9}{2}.
\end{align*}
This finishes the proof of this lemma.
\end{proof}

\subsection{Derivation and justification of the normal form}

In this section, we derive and justify the normal form. We follow the steps of the proof in \cite{Pel}.

\subsubsection{Modulation equations for parameters $a$ and $c$}

%Let $X_c = {\rm span}\{\partial_{\xi} u_c, \partial_c u_c \}$ be an invariant subspace of $L^2(\mathbb{R})$ for the double zero
%eigenvalue of the linearized operator $\partial_{\xi} L_c : H^3(\mathbb{R}) \to L^2(\mathbb{R})$, according to the Jordan block (\ref{Kernel-double}). Thanks to the exponential decay of $u_c(\xi)$ as $|\xi| \to \infty$,
%there is $\mu_0 > 0$ such that $X_c$ is also an invariant subspace of $L^2_{\mu}(\mathbb{R})$ for
%$\partial_{\xi} L_c : H^3_{\mu}(\mathbb{R}) \to L^2_{\mu}(\mathbb{R})$ for $\mu \in (0,\mu_0)$.
%Similarly, $X_c^* = {\rm span}\{ u_c, \partial_{\xi}^{-1} \partial_c u_c \}$ is an invariant subspace
%of $L^2_{-\mu}(\mathbb{R})$ for the double zero eigenvalue of the adjoint operator
%$-L_c \partial_{\xi} : H^3_{-\mu}(\mathbb{R}) \to L^2_{-\mu}(\mathbb{R})$.
%Recall that the exponentially weighted space is defined in (\ref{exp-space})
%and
%$$
%\partial_{\xi}^{-1} u(\xi) := \int_{-\infty}^{\xi} u(\xi') d \xi'.
%$$
%In order to avoid confusion between spaces $L^2_{\mu}(\mathbb{R})$ and $L^2_{\mu}(\mathbb{R} \times \mathbb{T})$,
%we specify whether the spatial domain is $\mathbb{R}$ or $\mathbb{R} \times \mathbb{T}$ in the $L^2$ inner products and
%their induced norms.

We start by stating the modulation theory in the following lemma.

\begin{lemma}
\label{lem-decomposition}
 Let $T>0$. There exists $\eps_0 > 0$, $\mu_0 > 0$, and $C_0 > 0$ such that if
$u \in C([0,T),H^1(\mathbb{R} \times \mathbb{T}) \cap H^1_{\mu}(\mathbb{R} \times \mathbb{T}))$
with $\mu \in (0,\mu_0)$ is a solution to the ZK equation (\ref{ZK}) satisfying
\begin{equation}
\label{orbit-bound}
\eps := \inf_{a \in \mathbb{R}} \| u(x+a,y,t) - u_{c_*}(x) \|_{H^1(\mathbb{R} \times \mathbb{T}) \cap H^1_{\mu}(\mathbb{R} \times \mathbb{T})}
\leq \eps_0, \quad t \in \mathbb{R}_+,
\end{equation}
then there exist $a, c \in C([0,T))$ and $\tilde{u} \in
C([0,T),H^1(\mathbb{R} \times \mathbb{T}) \cap H^1_{\mu}(\mathbb{R} \times \mathbb{T}))$
such that the decomposition
\begin{equation}
\label{decomposition-lemma}
u(x,y,t) = u_{c(t)}(\xi) + \tilde{u}(\xi,y,t), \quad \xi = x -  a(t)
\end{equation}
holds with $\tilde{u}(t) \in [X_{c(t)}^*]^{\perp}$ for every $t \in \mathbb{R}_+$, where
\begin{equation}
\label{symplectic-orthogonality}
[X_{c(t)}^*]^{\perp} = \left\{ \tilde{u} \in L^2_{\mu}(\mathbb{R} \times \mathbb{T}) : \quad
\langle u_{c(t)}, \tilde{u} \rangle_{L^2(\mathbb{R} \times \mathbb{T})} =
\langle \partial_{\xi}^{-1} \partial_c u_{c(t)}, \tilde{u} \rangle_{L^2(\mathbb{R} \times \mathbb{T})} = 0
\right\}.
\end{equation}

%Moreover, $c(t)$ and $\tilde{u}(t)$ satisfies
%\begin{equation}
%\label{orbit-bound-after}
%|c(t) - c_*| +  \| \tilde{u}(t) \|_{H^1(\mathbb{R} \times \mathbb{T}) \cap H^1_{\mu}(\mathbb{R} \times \mathbb{T})}
%\leq C \eps, \quad t \in \mathbb{R}_+.
%\end{equation}
\end{lemma}

Note that the modulation parameters $a,c \in C^1(\mathbb{R}_+)$
satisfy the following system
\begin{eqnarray}
S \left[ \begin{array}{c} \dot{c} \\  \dot{a} - c \end{array} \right]
= \frac{1}{2\pi}  \left[ \begin{array}{c} \langle \partial_c u_c, \tilde{u}^3+3u_c\tilde{u}^2 \rangle_{L^2(\mathbb{R} \times \mathbb{T})} \\
\langle \partial_{\xi} u_c, \tilde{u}^3+3u_c\tilde{u}^2 \rangle_{L^2(\mathbb{R} \times \mathbb{T})}
\end{array} \right]\label{evolution-a-c}
\end{eqnarray}
with the coefficient matrix
\begin{equation}
\label{coefficeint-S}
S := \left[ \begin{array}{cc}
\frac{1}{2} (M'(c))^2 - \frac{1}{2\pi} \langle \partial_{\xi}^{-1} \partial^2_c u_c, \tilde{u} \rangle_{L^2(\mathbb{R} \times \mathbb{T})} &
P'(c) + \frac{1}{2\pi} \langle \partial_c u_c, \tilde{u} \rangle_{L^2(\mathbb{R} \times \mathbb{T})} \\
P'(c) - \frac{1}{2\pi} \langle \partial_c u_c, \tilde{u} \rangle_{L^2(\mathbb{R} \times \mathbb{T})} &
\frac{1}{2\pi} \langle \partial_{\xi} u_c, \tilde{u} \rangle_{L^2(\mathbb{R} \times \mathbb{T})} \end{array} \right],
\end{equation}
where $M(c) = \int_{\mathbb{R}} u_c(\xi) d\xi$ and $P(c) = \frac{1}{2} \int_{\mathbb{R}} u_c^2(\xi) d \xi$.
From the expression (\ref{line-soliton}), we obtain $M'(c) = 0$ and $P'(c) =1/\sqrt{c}$.

\subsubsection{A secondary decomposition for $c = c_*$}

By Lemma \ref{theorem-bifurcation}, if $c = c_* = \frac{1}{3}$ and $\mu > 0$ is sufficiently small,
then $Y_{c_*} = {\rm span}\{ \psi_* \}$ is an invariant subspace of $L^2_{\mu}(\mathbb{R})$ for
the simple zero eigenvalue of the linearized operator
$\partial_{\xi} (L_{c_*} + 1) : H^3_{\mu}(\mathbb{R}) \to L^2_{\mu}(\mathbb{R})$.
Similarly,  $Y_{c_*}^* = {\rm span}\{ \eta_* \}$
is an invariant subspace of $L^2_{-\mu}(\mathbb{R})$ for the simple zero eigenvalue of the adjoint
operator $-(L_{c_*} + 1)\partial_{\xi} : H^3_{-\mu}(\mathbb{R}) \to L^2_{-\mu}(\mathbb{R})$.
Let us point out the double degeneracy of the Fourier harmonics $e^{iy}$ and $e^{-iy}$
when general transverse perturbations are considered.

The following lemma states the secondary decomposition of the solution $\tilde{u}$
defined in the primary decomposition (\ref{decomposition-lemma}).

\begin{lemma}
\label{lem-decomposition-secondary}
Under the assumptions of Lemma \ref{lem-decomposition}, let
$\tilde{u} \in C(\mathbb{R}_+,H^1(\mathbb{R} \times \mathbb{T}) \cap H^1_{\mu}(\mathbb{R} \times \mathbb{T}))$
be given by the decomposition (\ref{decomposition-lemma}) and (\ref{symplectic-orthogonality}).
There exist $b \in C(\mathbb{R}_+)$ and $v \in
C(\mathbb{R}_+,H^1(\mathbb{R} \times \mathbb{T}) \cap H^1_{\mu}(\mathbb{R} \times \mathbb{T}))$
such that the decomposition
\begin{equation}
\label{decomposition-secondary}
\tilde{u}(\xi,y,t) = \left( b(t) e^{iy} + \bar{b}(t) e^{-iy} \right) \psi_*(\xi) + v(\xi,y,t),
\end{equation}
holds with $v(t) \in [Y_{c(t)}^*]^{\perp}$ for every $t \in \mathbb{R}_+$, where
\begin{equation}
\label{symplectic-orthogonality-single}
[Y_{c(t)}^*]^{\perp} = \left\{ v \in [X_{c(t)}^*]^{\perp} : \quad
\langle \eta_* e^{iy}, v \rangle_{L^2(\mathbb{R} \times \mathbb{T})}
= \langle \eta_* e^{-iy}, v \rangle_{L^2(\mathbb{R} \times \mathbb{T})} = 0
\right\}.
\end{equation}
\end{lemma}

%\begin{proof}
%The proof is straightforward thanks to the fact
%$\langle \eta_*, \psi_* \rangle_{L^2(\mathbb{R})} \neq 0$
%by (\ref{limiting-coefficient}).
%\end{proof}

We decompose the modulation parameters as
\begin{equation}
\label{decomposition-a-c}
a(t) = \int_0^t c(t') dt' + h(t), \quad c(t) = c_* + \delta(t).
\end{equation}
We denote $L_c = L_{c_*} + \Delta L_c$,
with $\| \Delta L_c \|_{L^{\infty}} \leq A |c-c_*|$
for $|c-c_*|$ sufficiently small and $A$ is a positive constant independent of $c$.

The correction term $v$ in the decomposition (\ref{decomposition-secondary}) satisfies the time evolution equation
\begin{eqnarray}
\nonumber
v_t & = & \partial_{\xi} (L_{c_*} - \partial_y^2 +  \dot{h} + \Delta L_c) v
+  \dot{h} \partial_{\xi} u_{c_*+\delta} - \dot{\delta} \partial_c u_{c_*+\delta} - (\dot{b}e^{iy} + \dot{\bar{b}} e^{-iy}) \psi_* \\
& \phantom{t} &
+ \partial_{\xi} \left(  \dot{h} + \Delta L_c \right) (be^{iy} + \bar{b} e^{-iy}) \psi_*\nonumber\\
\label{amplitude-evolution-1}
& \phantom{t} &
-  \partial_{\xi} \left( \left( (b e^{iy} + \bar{b} e^{-iy} ) \psi_* + v \right)^3+3 u_c \left( (b e^{iy} + \bar{b} e^{-iy} ) \psi_* + v \right)^2\right).
\end{eqnarray}

%The two constraints in (\ref{symplectic-orthogonality-single}) represent the symplectic orthogonality conditions,
%which specify uniquely the complex parameter $b$ in the secondary decomposition (\ref{decomposition-secondary}).
%Again, one can show that $b \in C^1(\mathbb{R}_+)$

For $b \in C^1(\mathbb{R}_+)$ the modulation equation is given by:
\begin{align}
\nonumber\dot{b} &\langle \eta_*, \psi_* \rangle_{L^2(\mathbb{R})} +
b \langle \psi_*, (\dot{h} + \Delta L_c) \psi_* \rangle_{L^2(\mathbb{R})} +
\frac{1}{2\pi} \langle \psi_* e^{iy}, \Delta L_c v \rangle_{L^2(\mathbb{R} \times \mathbb{T})}\\
& = \frac{1}{2\pi}  \langle \psi_* e^{iy}, \left[ (b e^{iy} + \bar{b} e^{-iy} ) \psi_* + v \right]^3+3u_c\left[ (b e^{iy} + \bar{b} e^{-iy} ) \psi_* + v \right]^2 \rangle_{L^2(\mathbb{R} \times \mathbb{T})}.
\label{evolution-b}
\end{align}
On the other hand, plugging (\ref{decomposition-secondary}) and (\ref{decomposition-a-c})
into the system (\ref{evolution-a-c}), we write the equations for $h$ and $\delta$:
\begin{align}
S \left[ \begin{array}{c} \dot{\delta} \\  \dot{h} \end{array} \right]
= \frac{1}{2\pi}
\left[ \begin{array}{c} \langle \partial_c u_c, \left[ (b e^{iy} + \bar{b} e^{-iy} ) \psi_* + v \right]^3+3u_c\left[ (b e^{iy} + \bar{b} e^{-iy} ) \psi_* + v \right]^2 \rangle_{L^2(\mathbb{R} \times \mathbb{T})} \\
\langle \partial_{\xi} u_c, \left[ (b e^{iy} + \bar{b} e^{-iy} ) \psi_* + v \right]^3+3u_c\left[ (b e^{iy} + \bar{b} e^{-iy} ) \psi_* + v \right]^2 \rangle_{L^2(\mathbb{R} \times \mathbb{T})}
\end{array} \right],\label{evolution-ac}
\end{align}
where $S$ in (\ref{coefficeint-S}) is written as
\begin{equation}
\label{coefficeint-S-S}
S := \left[ \begin{array}{cc}
 - \frac{1}{2\pi} \langle \partial_{\xi}^{-1} \partial^2_c u_c, v \rangle_{L^2(\mathbb{R} \times \mathbb{T})} &
P'(c) + \frac{1}{2\pi} \langle \partial_c u_c, v \rangle_{L^2(\mathbb{R} \times \mathbb{T})} \\
P'(c) - \frac{1}{2\pi} \langle \partial_c u_c, v \rangle_{L^2(\mathbb{R} \times \mathbb{T})} &
\frac{1}{2\pi} \langle \partial_{\xi} u_c, v \rangle_{L^2(\mathbb{R} \times \mathbb{T})} \end{array} \right].
\end{equation}
%The system (\ref{evolution-b}) and (\ref{evolution-ac}) determine the time evolution of the varying
%parameters $b$, $h$, and $\delta$, whereas the evolution
%problem (\ref{amplitude-evolution-1}) determines the correction term $v(t) \in  [Y_{c(t)}^*]^{\perp}$.

\subsubsection{Near-identity transformations}

%Because $u_{c_*}$ and $\psi_*^2$ are even functions of $\xi$, whereas $P'(c_*) \neq 0$,
%the modulation equations (\ref{evolution-b}) and (\ref{evolution-ac})
%yields the following balance at the leading order:
%\begin{equation}
%\label{estimates-on-b-h}
%\dot{b} = \mathcal{O}((|\delta| + |b|^2) |b|), \quad \dot{h} = \mathcal{O}(|b|^2),
%\end{equation}
%whereas the source terms in the evolution problem (\ref{amplitude-evolution-1}) are of the order
%of $\mathcal{O}(|b|^2)$.
%In what follows, we write out the leading-order terms provided that $\delta$ and $b$
%remain small for every $t \in \mathbb{R}_+$. Recall that the initial bound (\ref{initial-bound})
%yields $|\delta(0)| \leq \eps^2$ and $b(0) = \eps$, where $\eps \in (0,\eps_0)$ is a small parameter.
%Smallness of $\delta(t) := c(t) - c_*$ for every $t \in \mathbb{R}_+$ is proven in Section 3.3,
%see the bound (\ref{final-bound}). Smallness of $b(t)$ for every $t \in \mathbb{R}_+$
%is assumed in the bound (\ref{assumption-on-b}) and is proven here from the normal form (\ref{normal-form}).
%
%In order to derive the normal form (\ref{normal-form}), we use the near-identity transformations, which
%are very similar to the ones used in the proof of Lemma \ref{lemma-modulated-wave}.
%In particular, we will remove the $\mathcal{O}(|b|^2)$
%terms in the equation for $\dot{h}$ and $v_t$.
Now, we write the order $\mathcal{O}(|b|^2)$ in the correction term $v$ in the decomposition (\ref{decomposition-secondary}) as for \eqref{v-component}:
\begin{equation}
\label{decomposition-tertiary}
v(\xi,y,t) = \left( b(t)^2 e^{2iy} + \bar{b}(t)^2 e^{-2iy} \right) w_2(\xi) +
|b(t)|^2 w_0(\xi) + w(\xi,y,t),
\end{equation}
where $w_0$ and $w_2$ are the same solutions of the linear inhomogeneous equations (\ref{inhom-eq-1})
and (\ref{inhom-eq-2}), whereas $w$ satisfies the transformed evolution equation
\begin{align}
\label{amplitude-evolution-2}
w_t  = & \partial_{\xi} (L_{c_*} - \partial_y^2 +  \dot{h} + \Delta L_c) w
+  \dot{h} \partial_{\xi} u_{c_*+\delta} - \dot{\delta} \partial_c u_{c_*+\delta}  \\
\nonumber &
- (\dot{b}e^{iy} + \dot{\bar{b}} e^{-iy}) \psi_* - (2 b \dot{b} e^{2iy} + 2 \bar{b} \dot{\bar{b}} e^{-2iy}) w_2 - (\bar{b} \dot{b} + b \dot{\bar{b}}) w_0 \\
\nonumber &
+ \partial_{\xi} \left(  \dot{h} + \Delta L_c \right)
\left[ (be^{iy} + \bar{b} e^{-iy}) \psi_* + ( b^2 e^{2iy} + \bar{b}^2 e^{-2iy}) w_2 +
|b|^2 w_0 \right] \\
%\nonumber &
%-  \partial_{\xi} (b e^{iy} + \bar{b} e^{-iy} ) \psi_* \left[( b^2 e^{2iy} + \bar{b}^2 e^{-2iy}) w_2 + |b|^2 w_0\right] \\
\nonumber & -  \partial_{\xi} \left( \left( (b e^{iy} + \bar{b} e^{-iy} ) \psi_* + v \right)^3+3 u_c \left( (b e^{iy} + \bar{b} e^{-iy} ) \psi_* + v \right)^2\right).
\end{align}
On the other hand, we need to write the expansion of $\dot{h}$ with respect to $b$. This is given by the first equation in the system (\ref{evolution-ac}) which implies that
\begin{equation}
\label{normal-form-1a}
P'(c_*) \dot{h} = 6 |b|^2 \langle u_{c_*} \partial_c u_{c_*}, \psi_*^2 \rangle_{L^2} + \mathcal{O}(|b|^4)
= \frac{64}{3}(c_*)^\frac{3}{2} |b|^2 + \mathcal{O}(|b|^4).
\end{equation}
Since $P'(c_*) = 1/\sqrt{c_*}$ and $c_* = \frac{1}{3}$, we obtain
\begin{equation}
\label{decomposition-h-dot}
\dot{h} = \frac{64}{27} |b|^2 + \mathcal{O}(|b|^4)
\end{equation}
and
\begin{equation}
\label{decomposition-last}
w(\xi,y,t) = \frac{64}{27} |b|^2 \partial_{c} u_{c_*}(\xi) + \tilde{w}(\xi,y,t),
\end{equation}
where $\tilde{w}$ satisfied a transformed evolution equation without the $\mathcal{O}(|b|^2)$ terms
in the right-hand side of (\ref{amplitude-evolution-2}).
Substituting (\ref{decomposition-tertiary}), (\ref{decomposition-h-dot}), and (\ref{decomposition-last})
into the modulation equation (\ref{evolution-b}) yields to:
\begin{eqnarray}
\nonumber
\dot{b} \langle \eta_*, \psi_* \rangle_{L^2} & = &
6 |b|^2 b \langle \psi_*^2u_{c_*}, w_0 + w_2 \rangle_{L^2} +
\frac{128}{9} |b|^2 b \langle \psi_*^2, u_{c_*} \partial_c u_{c_*} \rangle_{L^2} + 3b|b|^2 \|\psi_*\|^4_{L^4} \\
\label{normal-form-1}
& \phantom{t} & - \frac{64}{27} b|b|^2 \|\psi_*\|^2_{L^2}- b \delta \langle \psi_*, L_{c_*}'  \psi_* \rangle_{L^2}
+ \mathcal{O}(\delta^2 |b| + |b|^5),
\end{eqnarray}
where $L_{c_*}':= 1-6u_{c_*} \partial_c u_{c_*}$
%is given by (\ref{derivative-operator})
 and we have used $\Delta L_c = L_{c_*}' \delta + \mathcal{O}(\delta^2)$
for $\delta = c - c_*$. Thus, the equation (\ref{normal-form-1}) takes the form
\begin{equation}
\label{normal-form-1-equiv}
\dot{b} \langle \eta_*, \psi_* \rangle_{L^2}  = \frac{73}{105} 2^{9}(c_*)^\frac{7}{2} |b|^2 b
+ 6|b|^2 b \langle \psi_*^2, w_2 \rangle_{L^2} +
16 (c_*)^\frac{3}{2}  b \delta + \mathcal{O}(\delta^2 |b| + |b|^5).
\end{equation}
Here, we have used the following identities
\begin{align*}
\|\psi_*\|^2_{L^2}&=\frac{16}{3} (c_*)^\frac{3}{2},\qquad \|\psi_*\|^4_{L^4}=\frac{2^9}{35} (c_*)^\frac{7}{2}, \quad 6\langle \psi_*^2, u_{c_*} \partial_c u_{c_*} \rangle_{L^2}=\frac{64}{3}(c_*)^\frac{3}{2}\\
 \langle \psi_*^2u_{c_*}, w_0 \rangle_{L^2}&=-\frac{2^{10}}{105} (c_*)^\frac{7}{2} \quad \langle \psi_*, L_{c_*}'  \psi_* \rangle_{L^2}=-16 (c_*)^\frac{3}{2}.\\
\end{align*}
%where $\tilde{w}_2$ is found from the solution of the linear inhomogeneous equation
%(\ref{inhomogeneous-eq-tilde-w-2}).

%The modulation equation (\ref{normal-form-1-equiv}) is not closed on $b$
%because $\delta$ is related to $|b|^2$ by the second equation
%of the system (\ref{evolution-ac}). In fact, this equation relates $\dot{\delta}$
%to $\mathcal{O}(|\dot{\bar{b}} b|) = \mathcal{O}(|\delta| |b|^2 + |b|^4)$, however,
%it yields $\delta = \mathcal{O}(|b|^2)$ after integration. In order to avoid
%integration of the second equation of the system (\ref{evolution-a-c}), we
%use the momentum conservation $Q(u) = Q(u_0)$, where the momentum $Q$ is given by
%(\ref{momentum}).
At this stage, we still need to write the expansion of $\delta$ with respect to $b$. To do so, we include (\ref{decomposition-lemma}), (\ref{decomposition-secondary}), (\ref{decomposition-a-c}), (\ref{decomposition-tertiary}), and (\ref{decomposition-last}) into the momentum
$$Q(u) :=\frac{1}{2} \int_{\R \times \T} u^2(x,y)\, dxdy, $$
in order to obtain
\begin{equation}
\label{expansion-Q}
Q(u) = 2 \pi \left[ P(c_* + \delta) + |b|^2 \| \psi_* \|^2_{L^2} + 16(c_*)^\frac{3}{2} |b|^2+ \mathcal{O}(|\delta| |b|^2 + |b|^4) \right].
\end{equation}
We have used the fact that:
%$$
%\langle u_{c_*}, w_0 \rangle_{L^2} = -\frac{16}{3} (c_*)^\frac{3}{2}\quad   \langle u_{c_*}, \partial_{c} u_{c_*} \rangle_{L^2} = \frac{1}{\sqrt{c}_*},
%$$
$$\langle u_{c_*}, w_0 \rangle_{L^2} + \frac{64}{27} \langle u_{c_*}, \partial_{c} u_{c_*} \rangle_{L^2}= -16 (c_*)^\frac{3}{2}$$
By the momentum conservation, we get
\begin{equation}
\label{expansion-Q-Q}
Q(u_0) = 2 \pi \left[ P(c_*) + \frac{\delta}{\sqrt{c_*}}  + \frac{56}{3} (c_*)^\frac{3}{2}|b|^2
+ \mathcal{O}(\delta^2 + |b|^4) \right],
\end{equation}
Hence, since $c_* = \frac{1}{3}$, we obtain
\begin{equation}
\label{normal-form-1e}
\delta = \delta_0 - \frac{56}{27} |b|^2 + \mathcal{O}(\delta_0^2 + |b|^4),
\end{equation}
where $\delta_0$ is a constant in $t$ determined by the initial data.

Substituting equation (\ref{normal-form-1e})
into the modulation equation (\ref{normal-form-1-equiv}) yields to:
\begin{equation}
\label{normal-form-2}
\dot{b} \langle \eta_*, \psi_* \rangle_{L^2}  =6
 |b|^2 b  \langle \psi_*^2u_{c_*}, w_2 \rangle_{L^2} +16(c_*)^\frac{3}{2} b \delta_0 + \frac{47}{105}(c_*)^\frac{3}{2} |b|^2 b +
\mathcal{O}(\delta_0^2 |b| + |b|^5).
\end{equation}
Defining $\delta_0 := c_+ - c_*$, using the explicit expression
(\ref{limiting-coefficient}) and (\ref{derivative-eigenvalue}),
and truncating (\ref{normal-form-2}) we obtain the normal form
\begin{equation}
\label{normal-form}
\dot{b} = \lambda'(c_*) (c_+ - c_*) b + \gamma |b|^2 b, \quad t \in \mathbb{R}_+,
\end{equation}
with
\begin{equation}
\label{gamma}
\gamma := \frac{1}{8 c_*} \left( \langle \psi_*^2u_{c_*}, w_2 \rangle_{L^2} + \frac{47}{105}(c_*)^\frac{3}{2} \right).
\end{equation}
We recall that equation (\ref{inhom-eq-2}) can be written as $w_2= \frac{3}{2} (L_{c_*} + 4)^{-1} ( u_{c_*} \psi_*^2).$ Thus
$$\langle \psi_*^2u_{c_*}, w_2 \rangle_{L^2} =\frac{3}{2} \langle \psi_*^2u_{c_*},  (L_{c_*} + 4)^{-1} ( u_{c_*} \psi_*^2) \rangle_{L^2}  . $$
Since $L_{c_*} + 4 : H^2(\mathbb{R}) \to L^2(\mathbb{R})$ is strictly positive, we infer that
$$\langle \psi_*^2u_{c_*}, w_2 \rangle_{L^2} >0. $$
Hence, we deduce that $\gamma>0$.

\begin{remark}
Since $\gamma>0$, the zero solution to \eqref{normal-form} is unstable.
\end{remark}

\appendix
\section{The quadratic case.}
For k=1, Pelinovsky used in \cite{Pel} numerical simulations in order to show that $\gamma<0$. We provide a complete proof of that  claim. First, we recall that in the quadratic case $\gamma$ is given by
$$\gamma:=12\langle \psi_*^2, w_0 + w_2 \rangle_{L^2} +
144 \langle \psi_*^2, \partial_c u_{c_*} \rangle_{L^2} -48 \|\psi_*\|^2_{L^2},$$
where $w_{0}$ and $w_{2}$ are solutions to the following linear inhomogeneous equations:
$$
L_{c_{*}} w_{0}=12 \psi_{*}^{2}
$$
and
$$
\left(L_{c_{*}}+4\right) w_{2}=6 \psi_{*}^{2}
$$
with
$$L_c := -\partial^2_{x} +  c - 2 u_c ,\quad \psi_{*}(\xi)=\operatorname{sech}^{3}\left(\sqrt{c_{*}} \xi\right) \quad {\rm and} \quad  c_*= \frac{1}{5}.$$ From \cite{Pel}, we recall that
$$\langle \psi_*^2, w_0 \rangle_{L^2} = \frac{-160}{21\sqrt{c}_*}, \quad   \langle \psi_*^2, \partial_c u_{c_*} \rangle_{L^2} = \frac{4}{5 \sqrt{c}_*}, \quad {\rm and} \quad  \|\psi_*\|^2_{L^2}= \frac{16}{15 \sqrt{c}_*}.$$
On the other hand, we know that
$$w_2= 6(L_{c_*}+4)^{-1}(\psi_*^2).$$
Thus, from the bound $\|(L_{c_*}+4)^{-1}\| \leq 1/3 ,$ we have
$$\langle \psi_*^2, w_2 \rangle_{L^2} \leq 2  \|\psi_*^2\|^2_{L^2}= \frac{1024}{693 \sqrt{c}_*}. $$
Finally, we conclude that
$$\gamma\leq 12\left(-\frac{160}{21\sqrt{c}_*} + \frac{1024}{693 \sqrt{c}_*} \right) +  \frac{576}{5 \sqrt{c}_*} - \frac{256}{5 \sqrt{c}_*}= -12\frac{4256}{693 \sqrt{c}_*} + \frac{64}{ \sqrt{c}_*}=- \frac{6720}{693 \sqrt{c}_*}<0.$$

\bigskip

\noindent Declarations:\\

\noindent{\bf Ethical Approval:} The work is original and has not been submitted elsewhere.\\
{\bf Competing interests:} The authors declare that they have no competing interests.\\
{\bf Author contribution:} The authors declare that the study was realized in collaboration with the same contribution.\\
All authors read and approved the final manuscript.\\
{\bf Funding} is Not applicable.\\
Availability of data and materials: Not applicable.

\newpage
\noindent
Yakine Bahri
\\
Department of Mathematics and Statistics
\\
University of Victoria
\\
3800 Finnerty Road, Victoria, B.C., Canada V8P 5C2
\\
E-mail: ybahri@uvic.ca

\vspace{0.5cm}

\noindent
Hichem Hajaiej
\\
Department of Mathematics
\\
California State University, Los Angeles, USA
\\
5151 State Drive, 5151 Los Angeles, 90331 California, USA
\\
E-mail:hhajaie@calstatela.edu

\end{document}